\documentclass[12pt,a4paper]{amsart}

\usepackage{amsfonts,amsthm,epsfig,verbatim}
\usepackage{amsmath}
\usepackage{amssymb}
\usepackage{graphicx,color}
\usepackage{indentfirst}
\usepackage{float}

\usepackage{epsfig}
\usepackage{verbatim}
\usepackage{graphicx}

\title[Porous media equation]{Lack of uniqueness for weak solutions of the incompressible porous media equation}
\author{Diego Cordoba, Daniel Faraco and  Francisco Gancedo}

\date{December 16, 2009}

\address{Diego C\'ordoba, Instituto de Ciencias Matem\'aticas, Consejo Superior de Investigaciones Cientificas, Calle Serrano 123, 28006 Madrid, Spain.} \email{dcg@icmat.es}
\address{Daniel Faraco, Instituto de Ciencias Matem\'aticas CSIC-UAM-UCM-UC3M and Department of Mathematics, Universidad Aut\'onoma de Madrid, 28049 Madrid, Spain.} \email{
daniel.faraco@uam.es}
\address{Francisco Gancedo,
Department of Mathematics, University of Chicago, 5734 University Avenue, Chicago, IL 60637, U.S.A.} \email{fgancedo@math.uchicago.edu}

\begin{document}

\sloppy


\theoremstyle{plain}
\newtheorem{theorem}{Theorem}[section]
\newtheorem{lemma}{Lemma}[section]
\newtheorem{cor}{Corollary}[section]
\newtheorem{Prop}{Proposition}[section]
\theoremstyle{definition}
\newtheorem{Def}{Definition}[section]
\newtheorem{remark}{Remark}[section]

\newtheorem{Prob}{Problem}

\def\halmos{{\ \vbox{\hrule\hbox{\vrule
 height1.3ex\hskip0.8ex\vrule}\hrule}}\par \medskip}



\def \t{\theta}
\def \l{\lambda}
\def\e{\varepsilon}
\def \g{\nabla}

\def \D{\mathbb D}
\def\U{\mathcal{U}}
\def\F{\sharp ({\R})}
\def\L{\mathcal{L}}
\def\C{{\mathbb{C}}}
\def\N{{\mathbb{N}}}
\def\H{\mathcal{H}}
\def\ES{\mathbb E}
\def\M{\mathcal{M}}
\def\R{\mathbb{R}}
\def \O{\mathcal{O}}
\def\Tp{\mathbb{T}}

\def \cl{\textrm{cl }}
\def \i{\textrm{int }}
\def \c{\partial_{ z}}
\def \a{\partial_{\overline z}}
\newcommand{\rank}{\mathrm{rank}}
\newcommand{\di}{{\rm div}\thinspace}
\newcommand{\curl}{{\rm curl}\thinspace}
\newcommand{\grad}{\nabla}


\def\w {W^{1,p}}
\def\2x2{{\R^{2 \times 2}}}
\def\p{W^{1,p}(Q ,\bf{R}^2)}
\def\h{W^{1,2}(\Omega,\bf{R}^2)}
\def\mxn{{\R^{m \times n}}}
\def\loc{\textrm{loc}}
\def \sym{{\R_{\textrm{sym}}^{2 \times 2}}}
\def \s {\buildrel  {\star} \over {\rightharpoonup} }
\def \W{\rightharpoonup}
\def\se{\{g_j(x)\}_{j=1}^{\infty}}
\def\su{\{f_j\}_{j=1}^{\infty}}
\def\de{\{Df_j(z)\}_{j=1}^{\infty}}
\def\T{\textrm{Tr}}
\def\E{E_{\{k,-k\}}}
\def\k{E_{K^{-1}}}
\def\x{B_\infty}
\def\spt{\mathrm{spt}\,}
\newcommand{\dist}{\mathrm{dist}\,}

\def \S{Df_\sharp ( \mathcal{L}_\Omega^n)}
\def \y{$\w$-GYM }
\renewcommand{\div}{\mathrm{div}\,}
\newcommand{\tr}{\mathrm{Tr}}
\newcommand{\supp}{\mathrm{supp}}


\maketitle

\begin{abstract}
In this work we consider weak solutions of the incompressible 2-D
porous media equation. By using the approach of De Lellis-Sz\'ekelyhidi we prove
non-uniqueness for solutions in $L^\infty$ in space and time.

\end{abstract}

\maketitle

\section{Introduction}

The incompressible 2-D porous media equation (IPM) is described by
$$
\rho_t+\nabla\cdot (v\rho)=0
$$
where the scalar $\rho(x,t)$ is the density of the fluid. The
incompressible velocity field
$$
\nabla\cdot v = 0
$$
is related with the density by the well-known Darcy's law
\cite{bear}
$$
\frac{\mu}{\kappa}v=-\nabla p-(0,g\rho)
$$
 where $\mu$ represents the viscosity of the fluid, $\kappa$ is the permeability of the medium, $p$ is the pressure of the fluid and $g$ is acceleration due to gravity. Without lost of generality we will consider $\mu/\kappa=g=1$.

 In this paper we study the weak formulation of this system
and we construct non trivial solutions with $\rho, v\in
L^\infty(\Tp^2\times [0,T])$ from initial data $\rho(x,0)=0$. Here
$\Tp^2$ is the two dimensional flat torus.

We define a weak solution of IPM $(\rho,v,p)$ if
$\forall\varphi,\chi,\lambda_1,\lambda_2\in C_c^{\infty}([0,T)\times\Tp^2)$ with $\lambda=(\lambda_1,\lambda_2)$
the following identities hold
\begin{eqnarray}\label{defweak}
 \int\limits_0^T\int\limits_{\Tp^2}\rho(x,t)
\left(\partial_t\varphi(x,t)+
v(x,t)\cdot\nabla\varphi(x,t)\right)dxdt \\ +
\int\limits_{\Tp^2}\rho_0(x)\varphi(x,0)dx=0,\nonumber
\end{eqnarray}
\begin{equation}\label{defweak2}
\int\limits_0^T\int\limits_{\Tp^2}v(x,t)\cdot\nabla \chi(x,t)dxdt=0,
\end{equation}
\begin{equation}\label{defweak3}
\int\limits_0^T\int\limits_{\Tp^2}(v(x,t)+\grad p(x,t)+(0,\rho(x,t)))\cdot\lambda(x,t) dxdt=0.
\end{equation}

For initial data in the Sobolev class $H^s(\Tp^2)$ ($s>2$) there
is local-existence and uniqueness of solutions in a classical
sense and global existence is an open problem \cite{CGO}. It is
known the existence of weak solutions, where the motion takes
place in the interface between fluids with different constant
densities, modeling the contour dynamics Muskat problem \cite{CG}.
The existence of weak solutions for general initial data is not
known. In this context we emphasize that the solutions we construct satisfy
$$
\limsup_{t\rightarrow 0^+}\|\rho\|_{H^s}(t)=+\infty
$$
for any $s>0$ (see Remark \ref{irregularity}).

From Darcy's law and the incompressibility of the fluid we can write the velocity as a singular integral operators with respect to the density as follows
\begin{eqnarray*}
    v(x,t)= PV \int_{\R^2} \Omega(x-y)\,\rho(y,t)dy\, - \frac{1}{2}\left(0,\rho(x)\right),\quad x\in \R^2,
\end{eqnarray*}
where the kernel is a Calderon-Zygmund type
$$
\Omega(x)=\frac{1}{2\pi}\left(-2\frac{ x_1
x_2}{|x|^4},\frac{x_1^2-x_2^2}{|x|^4}\right).
$$
The integral operator is defined  in the Fourier side by
\begin{eqnarray*}
\widehat{v}(\xi)=
(\frac{\xi_1\xi_2}{|\xi|^2},-\frac{(\xi_1)^2}{|\xi|^2})\widehat{\rho}(\xi).
\end{eqnarray*}

This system is analogous to the 2-D surface
Quasi-geostrophic equation (SQG) \cite{CGO}, in the sense that is
an active scalar that evolves by a nonlocal incompressible velocity given by
singular integral operators. It follows that, for Besov spaces, if
the weak solution $\rho$ is in $L^3([0,T]\times B_3^{s,\infty})$
with $s>\frac13$ then the $L^2$ norm of $\rho$ is conserved
\cite{JWu}. This result frames IPM in the theory of Onsager's
conjecture for weak solutions of 3-D Euler equations
\cite{CET},\cite{Onsager}. However there is an extra cancelation, for SQG, due to the symmetry of the velocity
given by
\begin{eqnarray*}
\widehat{v}(\xi)=
i(-\frac{\xi_2}{|\xi|},\frac{\xi_1}{|\xi|})\widehat{\rho}(\xi)
\end{eqnarray*}
that provides global existence for weak solution with initial data
in $L^2(\Tp^2)$ \cite{Resnick}. Furthermore, one can find a
substantial difference between both systems for weak solutions of
constant $\rho$ in complementary domains, denoted in the
literature as patches \cite{Majda}. For IPM the Muskat problem
presents instabilities of Kelvin-Helmholtz's type \cite{CG} and
there is no instabilities for SQG (\cite{Rodrigo},\cite{G}).

The first results of non-uniqueness for the incompressible Euler
equations are due to Scheffer \cite{Scheffer} and Shnirelman
\cite{Shnirelman} where the velocity field is compacted supported
in space and time with infinite energy. The method of the proof we
use in this paper is based on understanding the equation as a
differential inclusion in the spirit of Tartar
\cite{Tartar79,Tartar82}. In a ground breaking recent paper De
Lellis-Sz\'ekelyhidi \cite{DeLellisSzekelyhidi1} showed that with
this point of view the modern methods for solving differential
inclusions  could be reinterpreted and adapted to construct wild
solutions to the Euler equation with finite energy. Our plan was
to investigate the scope of this approach in the context of the
porous media equation. However it turned out that the situation is
different and we have to take different routes at several places
which might be of interest for the theory of differential
inclusions. We describe them shortly. Firstly, using the
terminology of this area that will be recovered in the first
section, the special role of the direction of gravity yields
certain lack of symmetry in the wave cone $\Lambda$. Moreover the
set $K$ describing the non linear constraint belongs to the zero
set of a $\Lambda$ convex function. Then it follows  that,
opposite to Euler, the $\Lambda$ convex hull does not agree with
the convex hull and more relevant $K \subset\partial K^{\Lambda}$.
This is  an obstruction for the available  versions of convex
integration, the ones based on  Baire category
\cite{DacorognaMarcellini97,Kirchheim03,MuSy01} and direct
constructions \cite{Gromov, KMS02,MullerSverak03}. Thus we do not
see anyway to use the method to produce solutions to the equation
with constant pressure and velocity for some time. Surprisingly,
an easy argument shows that the states in $\rm{int}(K^\Lambda)$
can still be used to produce periodic weak solutions starting with
$\rho=v=0$. This leaves the difficulty of choosing a proper subset
$\tilde{K} \subset K $ with sufficient large $\Lambda$ hull. In
general the computation of $\Lambda$ hull might be rather
complicated. In this work  we argue differently to suggest a more
systematic method:  Instead of fixing a set and computing the
hull, we pick a reasonable matrix $A$ and compute $(A+\Lambda) \cap
K$. Then by \cite[Corollary 4.19]{Kirchheim03} it is enough to
find a set $\tilde{K} \subset (A+\Lambda) \cap K$ such that $A \in
\tilde{K}^{c}$ to find what are called degenerate $T4$
configurations supported in $\tilde{K}^c$. To our knowledge this
is the first time they are used to produce exact solutions. Then,
we are able to choose the set $\tilde{K}$ carefully so that the
construction is stable and allow us to solve the inclusion. Our
proof of this later fact is related, to the known ways of solving
differential inclusions based on $T4$ configuration
\cite{Szpolyconvex, KMS02, Kirchheim03, MullerSverak03} but we
have made an effort to keep the whole paper self contained. In
particular we do not appeal  to the general theory of laminates
\cite{Pedregal93,Pedregalbook, MullerSverak03} to minimize the
amount of new terminology and make the construction more explicit.

The paper is organized as follows. In Section 2 we review the
De Lellis-Sz\'ekelyhidi   approach to obtain weak solutions to a
nonlinear PDE . Section 3 is devoted to put IPM in this framework, the computation of the wave cone and
the corresponding potentials. Section 4 is devoted to the geometry
part of the construction and Section 5 to the analytic one. In
this last section we include some remarks about the nature of our
solution as well as why this approach seems to fail for
constructing  weak solutions compactly supported in space and
time.

\section{General strategy and notation}
 The strategy to construct wild solutions to a non linear PDE
, understanding it as a differential inclusion, builds upon the following main steps.
Consider the equation\footnote{Here $n$ and $m$ are
suitable spaces dimensions that we keep general at this point.}
 $$P(u)=0,\qquad u:\R^n\rightarrow \R^m,$$
where $P$ is a general differential operator.

\textbf{Tartar framework:} Decompose the non linear PDE into a
linear equation (Conservation law) $\L(\tilde{u})=0$  and a
pointwise constraint $\tilde{u}(x) \in K$. Possibly this
requires to introduce new state variables $\tilde{u}$ to linearize
the  non linearities

\[ P(u)=0 \iff \begin{cases} \mathcal{L}(\tilde{u})=0 \\
\tilde{u}(x) \in K \end{cases}. \]

Specially easy to bring to this framework are equations of the
form

\[\mathcal{L}(F_i(u))=0 \]

where $F_i:\R^m\to \R$ $i=1,...,l$ are
some functions possibly non linear. Then one considers new states
variables $\tilde{u}_i=F_i(u),$ which clearly satisfy the
conservation law. Then the pointwise constraint amounts to

\[ q_i=F_i(u).\]

 \textbf{Wave Cone:} As discovered by Tartar \cite{Tartar79} in his development of
compensated compactness to such conservation law one can associate
a corresponding wave cone $\Lambda_{\L},$ defined as follows:
\[\Lambda_{\L}=\{A\in \R^{m}\times\R^l :\, \exists \xi  \textrm{ such that for }  z(x)=Ah(x \cdot
\xi)\,,\L(z)=0 \},\]
where $h:\R\rightarrow\R$ is an arbitrary function. That is, the wave cone corresponds to one
dimensional solution to the conservation law.

\textbf{Potential:} If $A \in \Lambda_{\L}$ we can build $z:\R^n \to
\R^{m}\times\R^l$ such that

\begin{equation*}
\mathcal{L}(z)=0   \textrm{ and } z(\R^n) \subset \sigma_A
\end{equation*}
where $\sigma_A=\{t A:t \in \R \}$ is the line passing through $A$.
The wild solutions are made by adding one dimensional oscillating
functions in different directions $A$. For that it is needed to
localize the waves. This is done by finding a suitable potential,
i.e.  a differential operator $D$ such that
\[ \mathcal{L}(D )=0, \]
and the plane wave solutions to the conservation law
can be obtained from the potential.

\textbf{Pointwise constraint:} In the pointwise constraint one adds the additional features
desired for the weak solution (fixed energy, fixed values, etc). This describes
a new set $\tilde{K}\subset K$. Then the solution is obtained by two
additional steps.

\begin{itemize}

\item Find a bigger set $\U$ so that there exists at least a solution
 to the relaxed system
\[ U_0 \in \U,\]
with the proper boundary conditions. In
practice the set $\U \subset \textrm{int} (K^\Lambda)$, the
so-called $\Lambda$ convex hull of $K$. \item Then one fixes a
domain $\Omega$ compact in space and time and in properly chosen
subdomains of $\Omega$ one adds infinitely many small
perturbations to find a function $U_\infty \in K$ a.e $(x,t) \in
\Omega$ and $U_\infty=U_0 $ outside of $\Omega$. If $U_0$ is a
solution to the conservation law, so is $U_\infty$. Notice that then $U_0
\in K$ and to be sure that $U_\infty \neq U_0$, we need that $ U_0
\notin \tilde{K}$.

\end{itemize}

There is no standard way to produce the weak solutions from
$K^\Lambda$ and it is not even always possible. In \cite{DeLellisSzekelyhidi1} (for Euler equations)
it is enough that $(0,0) \in \mathcal{U}=K^{c}$, which is open in the
spaces where the images of the potential lie and that for every
state $U \in \mathcal{U}$, $A_{\tilde{u}} \in \Lambda$ such that
the segment which end points $\tilde{u} \pm A \in \U$ and

\begin{equation*}\label{mA}
|A| \ge C \textrm{dist}(\tilde{u},  K ).
\end{equation*}

As discussed in the introduction, this strategy fails in our
context and needs to be modified.

\section{Porous media equation}

Here we start by recalling the 2-D IPM system
\begin{equation}\label{IPM}
\begin{cases}
  \partial_t \rho+\nabla\cdot (v \rho)=0 \\
\nabla\cdot v=0 \\
 v=-\nabla p- (0,\rho).
\end{cases}
\end{equation}
Our task below is to adapt section 2 to the IPM system.
\subsection{Porous media equation in the Tartar framework}

Given functions $(\rho,v,q)$ we define the differential operator $\mathcal{L}$
by
\begin{equation}
\L(\rho,q,v)= \begin{cases}
   \partial_t \rho+ \nabla\cdot q \\
 \nabla\cdot v \\
 \textrm{Curl}(v + (0,\rho))\\
 \int_{\Tp^2} v \\
 \int_{\Tp^2} \rho
\end{cases}
\end{equation}

\begin{Prop}
A pair $(\rho,v)\in  L^\infty(\Tp^2\times [0,T],\R \times \R^2)$ is a weak
solution to IPM if and only if we find $u=(\rho,v,q)\in L^\infty(\Tp^2\times [0,T],\R \times \R^2\times \R^2)$ such that,

\[\begin{cases} \mathcal{L}(u)=0 \\ q=\rho v \end{cases}\]

\end{Prop}

{\it Proof.} The two first equations are obvious. The third follows
from the Hodge decomposition in $\Tp^2$. Namely consider a vector
field $f\in L^2(\Tp^2)$ with $\frac{1}{(2\pi)^2}\int_{\Tp^2} f
dx=0$ and $\curl f=0$. We claim that the Hodge decomposition
implies that $f$ is a gradient field. Namely, every field $f\in
L^2(\Tp^2)$ on the torus $\Tp^2$ has a unique orthogonal
decomposition
$$
f=\frac{1}{(2\pi)^2}\int_{\Tp^2} f dx +w+\grad p,
$$
such that $\di w=0$. But
 $\curl f=0=\curl w$. Now we have that $\di w=0$ and $\curl w=0$,
 then for example
 by using the formula $\curl\curl g=-\Delta g+\grad \di g$ one finds that $\Delta w=0$ and which yields $w=0$.
 Thus,
$$
f=\frac{1}{(2\pi)^2}\int_{\Tp^2} f dx +\grad p=\grad p,
$$

\halmos

It is convenient to write the differential part of $\L$ in a
matrix form. For that to each state $u=(\rho,v,q)$ we associate a
matrix value function
$$U(u):\R\times\R^2\times\R^2  \to \mathbb{M}^{3 \times 3}$$ by
$$U(u)=  \begin{pmatrix} -v_2 -\rho & v_1, & 0 \\
                        v_1 & v_2   & 0 \\
                        q_1 & q_2   & \rho
                        \end{pmatrix}$$
and the subspace  $\mathbb{U}\subset \mathbb{M}^{3 \times 3}$ by $ \mathbb{U}=U(\R^2 \times \R^2 \times \R)$.

The following lemma is straightforward.

\begin{lemma} Let $U=U(\rho,v,q) \in L^2(\Tp^2\times \R)$ with $\frac{1}{(2\pi)^2}\int_{\Tp^2} \rho
dx=\frac{1}{(2\pi)^2}\int_{\Tp^2} v dx=0$. Then the following
statement holds in the sense of distributions:

$$\L(u)=0 \iff \textrm{Div}(U)=0,$$
 where $\textrm{Div}(a_{ij})=
\partial_j a_{ij}$
\end{lemma}

\subsection{The wave cone}

We denote by $\Lambda$ the wave cone related to (IPM).

The reason to write the states in matrix form is that the wave
cone is particularly easy to characterize. Namely, we want a plane wave
$z(x)=A h(x \cdot \xi)$ to satisfy

\[Div (z)=0 .\]

We obtain that for each $i$, $h'(x \cdot \xi) \xi_j a_{ij}=0$,
that is

\[A (\xi)=0 .\]

In other words $A \in \Lambda$ if and only if $A$ is a singular matrix. Thus $A \in
\Lambda \cap \mathbb{U}$ if

\begin{equation*}\label{wavecone}
-\rho ( v_2^2+v_1^2+\rho v_2)=0 \iff \rho=\frac{-|v|^2}{v_2}
\textrm{  or } \rho |v|^2=0. \end{equation*}

Observe that fixed $\rho$ this is

\[ v_2^2+v_1^2+\rho
v_2=v_1^2+(v_2+\frac{\rho}{2})^2-\frac{\rho^2}{4} \]

Thus fixed $\rho$ the corresponding $v \in S(
(0,-\frac{\rho}{2}),\frac{|\rho|}{2})$.

\subsection{Potentials}

The aim of this section is to find a differential operator $D$
such that $\mathcal{L}(D)=0$. We define for $\psi,\varphi: \Tp^2 \times \R \to \R$,

\[D(\varphi, \psi)=\begin{pmatrix} \partial_{x_1x_1} \psi
&\partial_{x_1x_2} \psi &0 \\ \partial_{x_1x_2} \psi & -\partial_{x_1x_1} \psi
&0
\\ -\partial_{tx_1} \psi -\partial_{x_2} \varphi & -\partial_{t x_2} \psi
+\partial_{x_1} \varphi & \Delta \psi \end{pmatrix} -\begin{pmatrix}
\Delta \psi & 0 &0 \\ 0 & 0 & 0 \\ 0 &0 &0 \end{pmatrix} \]

Or more compactly written

\[
\begin{pmatrix} -\partial_{x_2x_2} \psi &\partial_{x_1x_2} \psi &0 \\
\partial_{x_1x_2} \psi & -\partial_{x_1x_1} \psi &0
\\ -\partial_{tx_1} \psi -\partial_{x_2} \varphi & -\partial_{t x_2} \psi
+\partial_{x_1} \varphi & \Delta \psi \end{pmatrix}\]

Clearly $D(\varphi,\psi) \in \mathbb{U}$.

\begin{lemma}
Let $\Omega \subset \Tp^2 \times \R$ be a bounded open set, $\varphi \in
W^{2,\infty}(\Omega)$ and $\psi\in W^{1,\infty}(\Omega)$. Then it holds that
\[ \mathcal{L}D(\varphi,\psi)=0 \textrm{ in } \Omega \]
in the sense of distributions.
\end{lemma}

\begin{proof}

In our notation we need to obtain that
\[Div(D(\varphi,\psi))=0 \]
distributionally. The claim follows since distributional partial
derivatives commute.  Namely, it holds that

\begin{description}
\item [I]  $ -\partial_{x_2x_2x_1} \psi+\partial_{x_1x_2x_2}
\psi=0$, \item [II] $\partial_{x_1x_2x_1} \psi
-\partial_{x_1x_1x_2} \psi=0$, \item [III] $(-\partial_{tx_1x_1}
\psi-\partial_{tx_2x_2} \psi +\partial_t \Delta \psi)+ (
-\partial_{x_2x_1} \varphi+\partial_{x_1x_2} \varphi)=0,$
\end{description}

as desired. Finally we observe that since the flat torus has no
boundary then
\begin{equation*}
\int_{\Tp^2} D(\varphi,\psi)=0
\end{equation*}
and thus $\mathcal{L}D(\varphi,\psi)=0$.

\end{proof}

\begin{remark}
Observe that we are solving $\Delta \psi=\rho$ and then using the
PDE to define the rest.
\end{remark}

\subsection{Plane waves with potentials}

In this section we show that for any direction $\Lambda$ with
$\rho\neq 0$ we can obtain a suitable potential. It will be more
convenient for us to work with saw-tooth  functions instead of trigonometric functions. The following
proposition is related to \cite[Proposition 3.4]{Kirchheim03}.

\begin{lemma}\label{buildingblock}
Let $U \in \Lambda$ with $\rho v \neq 0$, $0<\lambda<1,
\epsilon>0, \Omega \subset \Tp^2\times \R$. Then there exists a
sequence $u_{N}:\Omega \to \mathbb{U}$ of piecewise smooth
functions such that
\begin{itemize}
\item $\mathcal{L}(u_N)=0,$
\item $\displaystyle\sup_{(x,t) \in \Omega}
\mathrm{dist}(u_N(x,t),[-(1-\lambda)U, \lambda U])\le \epsilon,$
\item $|(x,t) \in \Omega: u_N(x,t)=(-1-\lambda)U|\ge \lambda (1-\epsilon),$
\item $|(x,t) \in \Omega: u_N(x,t)=\lambda U|\ge (1-\lambda) (1-\epsilon),$
\item $u_{N}\stackrel{*}{\rightharpoonup} 0$ in $L^\infty(\Tp^2\times\R).$
\end{itemize}
\end{lemma}

\begin{proof}

Let $S \in W^{2,\infty}(\R,\R)$ and set $$\psi_N(x_1,x_2,t)=\frac{1}{N^2}
S(N(-\frac{v_1}{\sqrt{|v_2|}}x_1+{\sqrt{|v_2|}}x_2 + c t))$$
and
$$\varphi_N(x_1,x_2,t)=\frac{1}{N} S'(N(-\frac{v_1}{\sqrt{|v_2|}}x_1+{\sqrt{|v_2|}}x_2 +
c t)).$$

Then

\[D(d\varphi_N, \psi_N)= S''(N(-\frac{v_1}{\sqrt{|v_2|}}x_1+{\sqrt{|v_2|}}x_2 +
c t))\begin{pmatrix} -v_2-\rho & v_1 &0 \\
v_1 & v_2 &0
\\ c \frac{v_1}{\sqrt{|v_2|}}-d\sqrt{|v_2|}  & - c \sqrt{|v_2|}- d\frac{v_1}{\sqrt{|v_2|}}
 & \rho \end{pmatrix}.\]

Now the matrix
\[
\begin{pmatrix}
\frac{v_1}{\sqrt{|v_2|}} & - \sqrt{|v_2|} \\
- \sqrt{|v_2|}& -\frac{v_1}{\sqrt{|v_2|}}
\end{pmatrix} \]
has determinant $-(v^2_1+v^2_2)/|v_2|\neq 0$ and hence defines a bijection of $\R^2$. Thus for
any $q_1,q_2$ we can choose $(c,d)$ so that
$$(\frac{v_1}{\sqrt{|v_2|}}-d\sqrt{|v_2|}, - c \sqrt{|v_2|}-
d\frac{v_1}{\sqrt{|v_2|}})=(q_1,q_2).$$ Therefore we have that for
any direction in $\Lambda$ with $\rho v \neq 0$ we can obtain a
suitable potential.

Next we choose an appropriate function $S$. We consider the Lipschitz 1-periodic functions $S,s:\R\rightarrow\R$ such that $S(0)=0$, $S'=s$, $s'(x)=1-\lambda-\chi_{[\lambda/2,1-\lambda/2]}(x)$ for $x\in [0,1]$. We have that $s(1/2)=0$, $s'(1/2+x)=s'(1/2-x)$ and $s(1/2+x)=s(1/2-x)$ if $|x|<1/2$.

\begin{figure}[H]
\begin{center}
\includegraphics[bb = 0 0 256 300, height=70mm,width=70mm]{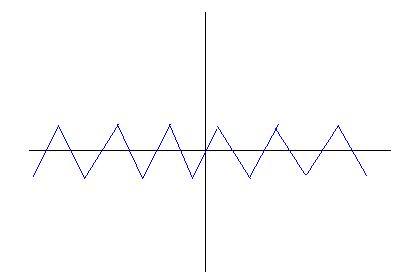}
\caption{{\it The graph of $s$.}} \vspace{5mm}
\end{center}
\end{figure}

We then localize the wave using the potential in the standard way i.e. we define a test functions $\zeta_{\epsilon'}$ such that
$$
|\zeta_{\epsilon'}|\leq 1, \qquad \zeta_{\epsilon'}=1 \textrm{ on } B_{1-\epsilon'}(0),\qquad \supp(\zeta_{\epsilon'})\subset B_1(0).
$$
Let $\Omega=B_1(0)$ and for a suitable $\epsilon'$ we define
$u_{N}^B=D(\zeta_{\epsilon'}(\varphi_N,\psi_N))$. For a general $\Omega$ consider
 disjoint balls $B_{r_k}({x}_k,t_k)$ such that
$$|\Omega\setminus\cup_k B_{r_k}({x}_k,t_k)|>1-\epsilon',$$
and finally we define
$$
u_N({x},t)=u_{N}^B(\frac{{x}-\overline{x}_k}{r_k},\frac{t-t_k}{r_k})\,\mbox{
on }\, x\in B_{r_k}({x}_k,t_k).
$$

\end{proof}

\section{The construction}

\subsection{Geometric setup}

 In this section we identify $\mathbb{U}$ with $\R\times\R^2\times \R^2$. For $A\in \mathbb{U}$ we will use coordinates $(\rho,w,z)$ with $\rho\in \R$ and $w,z\in \R^2$.  To manipulate the $\Lambda$ cone it will be helpful to introduce
the following notation. Along this section we denote by

\[S_\rho=S( (0,\frac{-\rho}{2}),\frac{|\rho|}{2}).\]
Similarly,
\[ B_\rho=B( (0,\frac{-\rho}{2}),\frac{|\rho|}{2}).\]

The set $K=(\rho,w,\rho w)\subset \mathbb{U}$ defines our pointwise constraint. The
 strategy in Section 2 requires to find
$X_0 \in (K\setminus {\tilde{K}}) \cap
\textrm{int}(\tilde{K}^{\Lambda})$. However
Proposition~\ref{boundary} shows that this is not possible.
Instead we will find   states of the form $(0,0,z) \in
\textrm{int}(\tilde{K}^{\Lambda})$. In the next section it will be
shown that this is enough to produce a weak solution.

We choose a suitable $\tilde{K} \subset K$. For that
we need to introduce the $T4$ configuration.

\begin{Def}\label{DefT4}
We say that $C\in \mathbb{U} $ is the center of   $T4$ configuration (degenerate)
if there exists $\{T_i(C)\}_{i=1}^4\in \mathbb{U}$ such that
\begin{itemize}
\item [a)]$C \in (\{T_i(C)\}_{i=1}^4)^c,$ \item [b)] $C-T_i(C) \in
\Lambda$.
\end{itemize}
We say that $A$ belongs to the $T4$ configuration if $A \in
[C,T_i(C)]$ for some $i$. In this case we also set $T_i(A)=T_i(C).$
\end{Def}

$T4$ has been the key pieces of subtle versions of convex
integration \cite{MullerSverak03,KirchheimPreiss03}. In this
degenerate form they appear for the first time in \cite{Kirchheim03}.

\begin{lemma}\label{T4}
Let $z_0 \neq (0,\frac{-1}{2}) \in B_1 $. Then there exists
$\delta=\delta(z_0)$ such that every
$A\in \mathbb{U}$ with $|A-(0,0,z_0)|\le \delta$ is the center of mass
of a $T4$ configuration in $K$. The  $T_i(A)$ are of the form
$(1,x_1,x_1),(1,x_2,x_2),(-1,y_1,-y_1),(-1,y_2,-y_2) \in K$, where
$(x_1,x_2,y_1,y_2)(A):\R^5 \to \R^2$ are differentiable
submersions. Moreover  $A \in \textrm{int} \{T_i(A)\}^c$.
\end{lemma}

\begin{proof}

First we observe that by the definition of the cone $\Lambda$

\[(\rho,w,z)-( 1,x,x) \in \Lambda \iff x-w \in S_{ 1-\rho}
\]
and
\[(\rho,w,z)-( -1,y,-y) \in \Lambda \iff y-w \in S_{ -1-\rho},
\]

 if in addition we have $$(\rho,w,z) \in
\textrm{int}(\{(1,x_1,x_1),(1,x_2,x_2),(-1,y_1,-y_1),(-1,y_2,-y_2)\}^c)$$ the
lemma is proved.

\begin{figure}[H]
\begin{center}
\includegraphics[bb = 0 0 256 300, height=80mm,width=80mm]{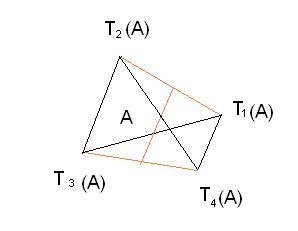}
\caption{{\it Convex hull of $T_i(A)$.}} \vspace{5mm}
\end{center}
\end{figure}

First, we notice that for each $(\rho,w,z)$  with $|\rho|<1$
there exists differentiable functions $x,y:\R \times \R^2\times \R^2
\to \R^2$ such that
\[(\rho,w,z) \in [(1,x,x), (-1,y,-y)].\]

Namely for $ t=\frac{1+\rho}{2}$ and $$x(A)=x(\rho,w,z)=\frac{z+w}{1+\rho},
\quad y(A)=y(\rho,w,z)= \frac{w-z}{1-\rho}, $$
 it holds
that   $$(\rho,w,z) = (t (1) +(1-t)(-1), t x+(1-t) (y), t x+(1-t)(-y)).$$

Now we consider the balls $ B_x=w+ B_{ 1-\rho},B_y=w+ B_{
-1-\rho}$ with centers $a_x,a_y$.

Thus, if $\delta \le \delta(z_0)$ is sufficiently small it follows
that \begin{equation}\label{xy} x=\frac{z+w}{1+\rho} \in
B_x\setminus a_x,\quad y=\frac{z-w}{1-\rho} \in B_y\setminus
a_y.\end{equation}

  In this case we can use the convexity of Euclidean balls to
show that  $x(A) \in [x_1,x_2], x_1,x_2 \in w+ S_{ 1-\rho}$ and
similarly for $y(A)$. Namely  if $|x-a|\le r$ then

\begin{equation*}x=
\frac{1}{2}(1-\frac{|x-a|}{r})\underbrace{(a-r\frac{x-a}{|x-a|})}_{x_1}+
\frac{1}{2}(1+\frac{|x-a|}{r})\underbrace{(a+r\frac{x-a}{|x-a|})}_{x_2}.\end{equation*}

The same argument provides us the mapping $y_1(A),y_2(A)$ Thus we
declare

\[\begin{aligned}&T_1(A)=(1,x_1,x_1),T_2(A)=(1,x_2,x_2), \\ &T_3(A)=(-1,y_1,-y_1),
T_4(A)=(-1,y_2,y_2).\end{aligned}\]

We have shown the existence of $\{\lambda_i\}_{i=1}^4$ such that
$0<\lambda_i<1$, $\sum_{i=1}^4 \lambda_i=1$ and
\[A=\sum_{i=1}^4 \lambda_i T_i(A). \]

 Since the $T_i(A)$ are linearly independent and $\lambda_i \neq
 0$ for $i=1,2,3,4$ we arrive to  the desired claim

 \[ A \in \textrm{int}(\cup_i T_i(A))^c. \]

We turn to the regularity of the mappings $T_i(A):\U \to \R^2$. It
is enough to consider the case $T_1$. Notice that from the convexity
argument follows
\begin{equation}\label{ar}
a_x=a(w,\rho)=w+(1-\rho)(0,\frac{-1}{2})\quad\text{and}\quad r=1-\rho .
\end{equation}
Thus, both are differentiable functions of $A$. Since $x \neq
a_x$ the same is true for $x_1,x_2, y_1$ and $y_2$. It remains to
see that they are submersions. It is enough to argue for the
function $x_1$. We write $x_1=F_1\circ F_2$ with $F_1:\R \to
\R^5$, $F_1(r,z,w) = (r_x,a_x,x)$ which are defined in
(\ref{ar},\ref{xy}) and $F_2:\R^5 \to \R^2$ is given by
$F_2(r,a,x) =a-r\frac{x-a}{|x-a|}$. For small $\delta$ $F_1$ is a
submersion and $F_2$ is always a submersion.
\end{proof}

\begin{remark}
It is instructive  to realize that by definition if $(\rho,w,z)$ belongs to  $K
\cap \Lambda$ then $z_2\le 0$ and thus $(0,0,0) \notin \textrm{int}(K
\cap \Lambda)^c$.  This shows that the above strategy never could work with $z=0$
(compare with Proposition~\ref{boundary}).
\end{remark}

  We will choose a set $\mathcal{U}$ adapted to the $T4$
  configuration.

\begin{Def}\label{firstlambda}
Let $K \subset \mathbb{U}$. The first Lambda convex hull of $K$,
$K^{1,\Lambda}$ is defined by
\[K^{1,\Lambda}=\{A \in \mathbb{U}: A=t B+(1-t)C, \quad B,C
\in K,\quad B-C \in \Lambda, t \in [0,1] \}. \]
\end{Def}

\begin{Def}\label{def} Let $z \neq (0,\frac{-1}{2})\in B_{1}$ and the
corresponding $\delta(z)$ from lemma~\ref{T4}. We denote by  $\mathbb{B}_z\in \mathbb{U}$
 the euclidean ball centered in $(0,0,z)$ and with radius $\delta(z).$
Furthermore  we declare
\begin{equation*}
\mathcal{U}_{z}=\cup_{i=1}^4\{ \mathbb{B}_z \cup
T_{i}(\mathbb{B}_z)\}^{1,\Lambda}\quad\text{and}\quad K_{z}=\cup_{i=1}^4\{
T_{i}(\mathbb{B}_z)\}\in K.
\end{equation*}
\end{Def}

\begin{lemma}\label{open}
The set $\mathcal{U}_{z}\setminus K $ is open in $\mathbb{U}$.
\end{lemma}
\begin{proof}

Let $C \in \mathbb{U}$ with $|C|\le \epsilon$ and $A_t\in \mathcal{U}_{z}\setminus K $. Our task is to find an $\epsilon$ so that $C_t= A_t + C\in \mathcal{U}_{z}\setminus K $. By definition  $A_t=t A
+(1-t)X_A$, where $A \in \mathbb{B}_z$, $0<t<1$ and $X_A\in T_{i}(\mathbb{B}_z)$ for some $i=1,2,3,4$. Moreover it holds that $A-X_A \in \Lambda$.

 It will be enough to show  that if $\epsilon$ is small enough, then there exists
  $X \in T_i(\mathbb{B}_z)$ such
that

\[C_t-X \in \Lambda\quad\text{and}\quad X+\frac{1}{t}(C_t-X)
\in \mathbb{B}_z. \]

We use coordinates
$A=(\rho,w,z)$, $X_A=(1,x_A,x_A)$, $A_t=(\rho_t,w_t,z_t),$ $C=(\rho_C,w_C,z_C)$.
Since $X_A - A \in \Lambda$ it holds that $X_A - A_t \in \Lambda$ as
well.   Thus, by the definition of $\Lambda$,
 $x_A-w_t \in S_{1-\rho_t}$. Therefore there exists $\xi$ with  $|\xi|=1$ such that
\[X_A =w_t + (1-\rho_t)[(0,-\frac{1}{2}) + \xi].\]
 Next recall that
 $C_t-(1,x,x ) \in \Lambda$ if $x-w_t-w_C \in
 S_{1-(\rho_t+\rho_C)}$. Set $x=w_t+w_C +(1-\rho_t-\rho_C)((0,\frac{-1}{2})+\xi)$ and
 notice that
 \begin{equation*}
|x-x_A|\le |w_C|+2|\rho_C| \le 3|C|\le 3 \epsilon.
 \end{equation*}

The function $x:A \to x(A)$ is a submersion by Lemma \ref{T4} and thus
$x(\mathbb{B}_z)$ is an open subset of $\R^2$. Hence, by choosing
$6\epsilon \le \textrm{dist}(x_A,\partial x( \mathbb{B}_z))$,  we
obtain that $(1,x,x) \in T_i(\mathbb{B}_z)$.

Now since $A=X_A + \frac{1}{t}(A_t-X_A)$ we have
\begin{eqnarray*}
|X+\frac{1}{t}(C_t-X)-A|&=&|X-T_i(A)|+
\frac{1}{t}|C_t-A_t|+|X-T_i(A)|\\&\le& (3+\frac{4}{t})\epsilon.
\end{eqnarray*}

Then if $\epsilon \le  \frac{1}{8} t\cdot\textrm{dist}(A,\partial
\mathbb{B}_z)$ it follows that $A_C=X+\frac{1}{t}(C_t-X) \in
\mathbb{B}_z$ and  \[C_t=tA_C+(1-t)X \in \mathcal{U}_{z}.\] A
final choice $ 8\epsilon \le \min\{t\cdot\textrm{dist}(A,\partial
\mathbb{B}_z),\textrm{dist}(x_A,\partial x( \mathbb{B}_z))\}$ yields the claim.
\end{proof}

\section{From geometric structure to weak solutions}

There are many ways to pass from  solutions in $\U$ to solutions
in $K$ depending on the geometry of the sets. We have followed to
present an approach which in some sense axiomatize the arguments
in \cite{DeLellisSzekelyhidi1}. However it is by no means the only
argument and for example one can verify that our sets $K$ and $\U$
verify the conditions in \cite[Proposition 4.42]{Kirchheim03}. The
argument is easier to explain via the following definition which
is related to the notion of stability near $K$ given in
\cite[Definition 3.15]{Kirchheim03}

\begin{Def}[Analytic Perturbation Property, APP]
The sets $K \subset \mathcal{U}$ have the analytic perturbation
property if for each $A \in \mathcal{U}$ and for every domain
$\Omega \subset \R^2_x \times \R$
there exists a sequence $Z_j$ of piecewise smooth function such
that

\begin{itemize}
\item [i)] $\mathcal{L}(Z_j)=0$.
 \item [ii)] $Z_j$ is supported in $\Omega$.
\item [iii)]$A+Z_j \in \mathcal{U}$ a.e.
 \item [iv)] $ Z_j \rightharpoonup 0$.
 \item [v)] There exists $c>0$ such that
  $\int_{\Omega}|Z_j|\geq c \cdot\textrm{dist}_K(A)|\Omega|.$
 \end{itemize}
\end{Def}

Next we prove that our sets enjoy the analytic perturbation
property. The reader that is familiar with convex integration will realize
that in practice our arguments are related to the concept  of  in
approximation \cite{MullerSverak03,Gromov} for $\Lambda$ convexity
(based on degenerated $T4$ configurations).
\begin{lemma}\label{KUAPL}
The sets $K_z$ and $\mathcal{U}_z$ have the APP property.
\end{lemma}

\begin{proof} We will obtain   v) as a consequence of
 \begin{itemize}
\item [$\acute{\text{v}}$)]There exists $c_1,c_0>0$ such that
  \[ |\{(x,t) \in \Omega: |Z_j(x,t)|\ge c_1\textrm{dist}_K(A)\}|\ge
  c_0 |\Omega|.\]
   \end{itemize}
  Let $C\in\mathbb{B}_z$. By lemma~\ref{T4} is the center
of a $T4$ configuration $ \{T_i(C)\}_{i=1}^4$. Now observe
that for $s$ small enough  $C$ belongs also to a $T4$
configuration supported in $T_{i,s}(C)=sC+(1-s)T_i(C) \in
\textrm{int} K_{z,\delta}$. Namely, since  $T_{i,s}(C) \to
{T_i}(C)$ as $s \to 0$ it follows that $C  \in
\textrm{int}(\{T_i(C)\})^c$ implies that $C \in \{T_{i,s}(C)\}^c$. The other properties are obvious.

Let
$T_{i,s}(C)$ be as above. First notice that since $|C-T_i(C)|\ge
\textrm{dist}(C,K_z)$ by continuity we can choose $s_0$ such that for $s>s_0$ it holds that

\begin{equation}\label{dist} 2|C-T_{i,s}(C)|\ge \textrm{dist}(A,K_z).
\end{equation}

Since $C \in \{T_{i,s}(C)\}^c$, there exist ${\lambda_i}>0$ with
$\sum \lambda_i=1$ such that  $C=\sum_i \lambda_i T_{i,s}(C)$. Since the $T4$ configuration is degenerate, if we set
$C_i^\epsilon=C+ \epsilon \sum_{j=1}^i({T_{i,s}}(C)-C)$ and
$T_i^\epsilon(C)=T_{i,s}(C)+ C_i^\epsilon$ it follows that
\[C_{i+1}^\epsilon= (1-\frac{\epsilon}{1+\epsilon})
C_{i}^\epsilon+ \frac{\epsilon}{1+\epsilon} T_{i+1}^\epsilon(C),\]
with $T_{i+1}^\epsilon(C)-C_{i}^\epsilon \in \Lambda$, i.e
$T_i^\epsilon$ is a non degenerate $T4$ (We drop the dependence
of $s$ for simplicity).

Fix a sequence $\eta_j$ such that $\Pi_{r=1}^\infty (1-\eta_r)\ge
\frac{1}{2}$ and $\varepsilon=\frac{\epsilon}{1+\epsilon}>0$. We will obtain the sequence
$\{Z_j\}_{j=0}^\infty$ recursively. We claim that given $Z_j$
there exists $Z_{j+1}$ such that
\begin{itemize}
\item [i)]$|\{(x,t):| Z_{j+1}(x)|\ge \frac{1}{4}
\textrm{dist}_K(C)| \}| \ge |\{(x,t):| Z_{j}(x)|\ge \frac14
\textrm{dist}_K(C)| \}|+(1-\varepsilon)^j \varepsilon \frac{1}{2}.$
\item [ii)] There exists $i=i(j)$ such that
\[ \Pi_{r=1}^j (1-\eta_r)(1-\varepsilon)^j\le |\{(x,t): C+Z_{j+1}(x)= C_i^\epsilon
\}|.\] \item [iii)]
$Z_j \rightharpoonup 0.$
\item [iv)] $C+Z_j \in \mathcal{U}_z.$
\end{itemize}

The proof follows by induction with $Z_0=0$. We start by considering
the open set $\Omega_j=\{(x,t) \in \Omega:
C+Z_{j}=C_i^{\epsilon}\}$. Since $C_{i}^\epsilon=
(1-\frac{\epsilon}{1+\epsilon}) C_{i-1}^\epsilon+
\frac{\epsilon}{1+\epsilon} T_{i}^\epsilon(C)$ we apply  the
Lemma~\ref{buildingblock} in $\Omega_j$ with
$\eta<\eta_{j+1}<\eta_0$ to obtain a new sequence $\{Z_{j,k}\}$.

\begin{itemize}
\item $Z_{j,k} \rightharpoonup 0.$ \item $|\{(x,t) \in \Omega_j:
Z_{j,k}+C_i^\epsilon=T_i^\epsilon({C})\}|\ge
(\frac{\epsilon}{1+\epsilon})(1-\eta_{j+1})|\Omega_j|.$ \item
$|\{(x,t) \in \Omega_j: Z_{j,k}+C_i^\epsilon={C}_{i-1}^\epsilon
\}|\ge (\frac{1}{1+\epsilon})(1-\eta_{j+1}) |\Omega_j|.$ \item
$\displaystyle\sup_{({x},t) \in \Omega_j}
\mathrm{dist}(Z_{j,k}({x},t)+C_i^\epsilon,[T_i^\epsilon({C}),{C}_{i-1}^\epsilon])\le
(1-\eta_j).$
\end{itemize}

Now set $Z_{j+1,k}=Z_{j}+  Z_{j,k} $ and
$\varepsilon=\frac{\epsilon}{1+\epsilon}$. Let us verify
properties $i)$ and $ii)$ . Notice that if   $(x,t) \in \Omega_j$
and $ Z_{j,k}(x)+C_i^\epsilon=T_i^\epsilon(C) $, then
$Z_{j+1,k}(x,t)= T_i^\epsilon({C})-C$ and thus by (\ref{dist}), for
small $\epsilon$, we get

\[|Z_{j+1,k}(x,t)| \ge \frac{1}{4} \textrm{dist}_K(C).\]

For $i)$ we have

\[\begin{aligned}
|\{(x,t):| Z_{j+1,k}(x,t)|\ge \textrm{dist}_K(C) \}| &\ge |\{(x,t)
\in \Omega \setminus \Omega_j:| Z_{j}(x,t)|\ge \textrm{dist}_K(C)
\}|\\&+ |\{(x,t)\in \Omega_j:
Z_{j,k}(x,t)+C_i^\epsilon=T_i^\epsilon(C)\}|
\\&\ge
|\{(x,t) \in \Omega \setminus \Omega_j:| Z_{j}(x,t)|\ge \textrm{dist}_K(C) \}|\\
&+\varepsilon (1-\eta_{j+1})|\Omega_j| \\ &\ge |\{(x,t) \in \Omega
\setminus \Omega_j:| Z_{j}(x,t)|\ge \textrm{dist}_K(C) \}|\\ &+
\varepsilon \Pi_{r=1}^{j+1}(1-\eta_r)(1-\varepsilon)^j
|\Omega|\\\ge \frac{1}{2} \varepsilon (1-\varepsilon)^j|\Omega|.
\end{aligned}\]

As for $ii)$, we have

\[ \begin{aligned}&|\{(x,t) \in \Omega: Z_{j+1,k}(x,t)+C ={C}_{i-1}^\epsilon  \}| \\
&\ge |\{(x,t) \in \Omega_j: Z_j(x,t)+Z_{j,k}(x,t)+C=
{C}_{i-1}^\epsilon| \ge (1-\varepsilon) (1-\eta_{j+1}) |\Omega_j|
\\& \ge (1-\varepsilon)^{j+1} \Pi_{r=1}^{j+1}
(1-\eta_{j+1})|\Omega|.\end{aligned}
\]

Next, property iv) follows because  the segments $[T_i^\epsilon({C}),{C}_{i-1}^\epsilon]$ are compact subsets
of $\mathcal{U}_z$ and by a diagonal argument we choose a subsequence $Z_{j,k(j)}$
such that

\[  Z_j \rightharpoonup 0. \]

It just remains to show that ${Z_j}$ satisfies property $\acute{\text{v}}$
 in the definition of analytic perturbation property (APP). Indeed by
 property i) in the definition of $Z_j$ yields

\[\begin{aligned} |\{(x,t) \in \Omega: |Z_{j+1}(x,t)|\ge
\frac{1}{4}\textrm{\textrm{dist}}_K(C)\}|&=
|\Omega|\sum_{k=1}^j (1-\epsilon)^j \epsilon \frac{1}{2}= |\Omega|\frac{1}{2}
(1-\epsilon^{j+1})\\ &\ge \frac{1}{4}|\Omega| \end{aligned}\]
for $j$ large enough. Thus for  $C \in \mathbb{B}_z$ (APP)
holds.

Let $A \in \mathcal{U}_z\setminus  \mathbb{B}_z$. By
definition  $A= sC+(1-s)X_C$ with $C\in \mathbb{B}_z$, $X_C\in T_i(\mathbb{B}_z)$, $0<s<1$ and  $ C- X_C\in\Lambda$.
 Consider the sequence $\{Z_j\}$
given by
 Lemma~\ref{buildingblock}. If $s<\frac{7}{8}$ $\{Z_j\}$ is the sequence required by the
 (APP) property. Hence we may assume that $s\ge \frac{7}{8}$, obtaining
that $|\{x \in \Omega: A+Z_j(x)=C\}|\ge \frac{7}{8}$. Let
$\Omega_j=\{(x,t) \in \Omega: A+Z_j(x,t)=C\}$. Since $C \in
\mathbb{B}_z$ we can use the argument above for the domain
$\Omega_j$. We are provided with  a sequence $\{Z_k\}$ such that
$Z_k \rightharpoonup
 0,Z_k+C \in \mathcal{U}_z$ and

\[
|\{(x,t) \in \Omega_j: |Z_k(x,t)|\ge
\frac{1}{4}\textrm{dist}_K(C)\}|\ge
  \frac{1}{4} |\Omega_j| \ge \frac{7}{32}|\Omega|. \]

A proper subsequence $W_j=Z_j+Z_{k(j)}$ satisfies all the desired
properties in the definition of (APP). As a matter of fact just
property $\acute{\text{v}}$ needs verification. Notice that for $C \in
\mathbb{B}_z$ and $X \in T_i(\mathbb{B}_z)$ it holds that
\[1-\delta<|C-X|\le 1 + \delta+\max_{C \in
\mathbb{B}_z}|T_i(C)| =M(\delta). \]

Thus $\textrm{dist}_K(C)> 1 - \delta$. Next we notice  that

\begin{equation}\label{A}
\textrm{dist}_K(A)\le |A-X_C|\le s |C-X_C|\le
M(\delta) .\end{equation}

 Let  $(x,t) \in \Omega_j$ be such
that $|Z_k(x,t)|\ge \frac{1}{4}\textrm{dist}_K(C)$. Then
\begin{equation}\label{B}\begin{aligned}
|W_j(x,t)| &\ge \frac{1}{4}\textrm{dist}_K(C)-|C-A|\\&\ge
  \frac{1}{4}(1-\delta)-(1-s)|C-X_C|\\&\ge \frac{1}{8}(1-\delta).
\end{aligned}
\end{equation}

Hence if we declare $c_1(\delta)=\frac{1-\delta}{8M(\delta)}$, and
we put together (\ref{A}) and (\ref{B}) we obtain

\[c_1(\delta)\textrm{dist}_K(A)\le |W_j(x,t)|\]

for every $(x,t) \in \Omega_j$ with $|Z_k(x,t)|\ge
\frac{1}{4}\textrm{dist}_K(C)$. We arrive to the estimate

\[|\{(x,t) \in \Omega:|W_j(x,t)|\ge c_1(\delta)\textrm{dist}_K(A)|\}\ge
\frac{7}{32}|\Omega|.\]

Finally (v')  holds with $c_1(\delta)=c_1$ and $c_0=\frac{7}{32}$.

\end{proof}

Next we show that if a set has a perturbation property there are
many  solutions for our inclusion. We formulate the existence
theorem in the following way:

\begin{theorem}\label{existence}
Let $K, \mathcal{U}$ be two sets that satisfy the Analytic
Perturbation Property and $\Omega \subset \Tp^2 \times \R$ a
bounded open set and $A \in \mathcal{U}$. Then there are infinitely
many solutions $U$ such that

\begin{enumerate}
\item $\mathcal{L}(U)=0$ \item $U(x,t) \in K$ a.e $(x,t) \in
\Omega$ \item $U(x,t)=A$ for $(x,t) \in \Tp^2 \times \R \setminus
\Omega$
\end{enumerate}
\end{theorem}

We shall prove the theorem using Baire category since it yields
infinitely many solutions. In addition we will provide a direct argument
to construct such solutions since it helps to understand its nature.

\begin{Def}[The space of subsolutions $X_0$]
 Let $A\in\mathcal{U}$. We say that $ U \in L^\infty(\Tp^2 \times \R)$ belongs to $X_0$
 if

\begin{itemize}
\item [(I)] $ U $ is
piecewise smooth.\hfill (Regularity)
 \item [(II)] For $t \le 0$ and $t\ge T$
 $ U(x,t)= A. $\hfill (Boundary conditions)
 \item [(III)]$\mathcal{L}(U)=0.$\hfill (Conservation law)
 \item [(IV)] $U(x,t) \in \mathcal{U}$ a.e $(x,t) \in
\Tp^2 \times \R.$ \hfill (Relaxed inclusion)
\end{itemize}
\end{Def}
We endow $X_0$ with the $L^\infty$ weak star topology and declare
$X=\overline{X_0}$.

If a set has the APP property then it holds the perturbation
lemma.
\begin{lemma}[Perturbation Lemma]
 There exists $C>0$ such that for every $U_0 \in X_0$
we can find  a sequence $U_k\in X_0$ such that
$U_k \rightharpoonup U$ but

\[\int_{\Tp^2 \times \R} |U_k-U_0|^2 \ge C
\int_{\Tp^2\times [0,T]} \textrm{dist}^2_K (U_0).\]
\end{lemma}

\begin{proof}

The sequence $U_k$ will differ from $U_0$ only on the compact set
$\Tp^2 \times [0,T]$.
 Since $U_0$ is piecewise smooth  we can find a
finite family of pairwise disjoint balls $B_j \Subset\Tp^2 \times
[0,T]$ such that
\begin{equation} \label{discretization}
\int_{\Tp^2\times [0,T]} \textrm{dist}_K (U_0) \le 2 \sum
\textrm{dist}_K (U_0(x_j,t_j)|B_j|\end{equation}

Next we use that $U_0(x_j,t_j) \in \mathcal{U}_z$ to apply the APP
property in the domain $B_j$. Let us called the corresponding
perturbation sequence  by $\{Z_{j,k}\}_{k=1}^\infty$ supported in
$B_j$. Thus we consider the sequence

\[U_k=U_0+\sum Z_{j,k},\]
which satisfies that $U_k \in \mathcal{U}$ (a small continuity
argument is needed).

Then
\[ \begin{aligned}\int_{\Tp^2\times \R} |U_k-U_0|= &\sum_j \int_{B_j} |Z_{j,k}|
\ge c\sum \textrm{dist}_K (U_0(x_j,t_j))|B_j|\\ &\ge \frac{c}{2}
 \int \textrm{dist}_K (U_0)
\end{aligned}\]
where the last inequality follows from (\ref{discretization}).

\end{proof}

Next we recall how to obtain a solution $U(x,t)\in K$ a.e. from the perturbation lemma and the non emptiness of $X_0$.

 The solution $U$ will be the strong limit of a sequence $U_k$
which is obtained from the perturbation lemma. Hence strong
convergence implies that $U$ is in $K$. The Baire category
argument yields the existence of such $U$ as point of continuity
of the Identity map and then a sequence ${U_k}$ which converge to
it. The direct argument builds the sequence ${U_k}$ iteratively
and then show that indeed it converges strongly to some $U$.\\

\textbf{ Baire Category:}

We consider the Identity as  a Baire-one map from $(X,weak*)$ to
$(X, L^1(\Tp^2\times \R))$ (see \cite[Lemma 4.5]{DeLellisSzekelyhidi1}).

\begin{lemma} Let $U \in X$ be a point of continuity of the
identity.  Then $U \in K \textrm{ a.e } (x,t) \in \Omega \times
[0,T].$
\end{lemma}

Suppose that $U \in X$ is a point of continuity of the identity.
By definition of $X$, there exists $\{U_j\} \in X_0$ such that
\[U_j \rightharpoonup U \]
in the weak star topology. As $u$ is a point of continuity of the
identity it holds that
\[\lim_{j \to \infty} \int_{\Tp^2\times \R} |U_j-U|=0 \]

and thus,
\[
\lim_{j \to \infty} \int_{\Tp^2\times [0,T]}
\textrm{dist}_{K}(U_j) dx= \int_{\Tp^2\times [0,T]}
\textrm{dist}_{K}(U).
\]

For each $U_j$ we use a perturbation lemma. The corresponding
sequence is indexed by $U_{j,k}$. By a diagonal argument we can
choose a sequence $U_{j,k(j)}$ such that
\[U_{j,k(j)} \rightharpoonup U \]
and since $U$ is a point of continuity yields

\[\lim_{j \to \infty} \int_{\Tp^2\times \R} |U_{j,j(k)}-U|=0 \]

Now

\[ \begin{aligned}
\int_{\Tp^2\times [0,T]} \textrm{dist}_{K}(U)  &= \lim_{j \to
\infty} \int_{\Tp^2\times [0,T]} \textrm{dist}_{K}(U_j)
\underbrace{\le}_{P.Lemma} C\lim_{j \to \infty} \int_{\Tp^2\times
\R_t} |U_j-U_{j,k(j)}| \\ &\le C\lim_{j \to \infty}
\int_{\Tp^2\times \R} (|U_j-U|+|U-U_{j,k(j)}|)  =0.
\end{aligned}\]

Proof of Theorem 5.1: It is a well known fact that the points of
continuity of a Baire one map are a set of second category (see
\cite{Oxtoby}. Therefore there are infinitely many $U \in X$ such
that $U\in K$ a.e. and $(x,t)\in \Omega$. This is equivalent to
property (II) in Theorem 5.1. Since (I) and (III) hold for any
element in $X_0$ they hold as well for any $U\in X$.\\

\textbf{Direct Construction:}

\begin{lemma}
To each function $U_0 \in X_0$ we can associate a sequence $\{U_j\}_{j=0}^{\infty}\in X_0$
such that there exists $U_\infty\in X$ satisfying
\begin{itemize}
\item  $U_\infty \in K$ a.e. $(x,t) \in\Tp^2 \times [0,T]$
 \item $\lim_{j \to \infty} \int_{\Tp^2\times \R}|U_j-U_\infty|=0$
\end{itemize}
\end{lemma}

\begin{proof}

We construct directly a function $U_\infty$ in $X$ and  a
sequence $U_k \in X$ which converges strongly to $U$ as follows.
We start with $U_0$ in the lemma and obtain  with $U_{k+1}$  from
$U_{k}$ by the perturbation lemma. To each $U_k$ we associate a
mollifier $\eta_{j(k)}$ such that

\begin{equation}\label{mollification}
\int_{\Tp^2\times \R} |U_k-\eta_{j(k)}*U_k|  \le 2^{-k}.
\end{equation}

Then we apply the perturbation lemma to the function $U$ to obtain
a corresponding  sequence ${U_s} \in X$. By the weak
convergence property we can choose $s(k)$ so that for each $i<k$

\begin{equation}\label{FWC} \int_{\Tp^2\times \R}
|U_k- U_{s(k)}* \eta_{j(i)})|\le 2^{-k}
\end{equation}

and finally we declare $U_{k+1}= {U_{s(k)}}$. Properties
(\ref{mollification}) and (\ref{FWC}) yield strong convergence.
Now

\[ \begin{aligned}
\int_{\Tp^2\times \R} \textrm{dist}_{K}(U)  &= \lim_{k \to
\infty} \int_{\Tp^2\times \R} \textrm{dist}_{K}(U_k)\\
&\underbrace{\le}_{P.Lemma} C \lim_{k \to \infty} \int_{\Tp^2\times
\R}|U_k-U_{k+1}| =0 \end{aligned}\]

\end{proof}

\begin{theorem}
For every $T>0$ there exists infinitely many non trivial weak solutions $(\rho, v)\in
L^\infty(\Tp^2\times [0,T])$ to the 2D IPM system \eqref{IPM} such that $\rho(x,0) =0$.
\end{theorem}

\begin{proof}
We fix $z\neq (0,-\frac12)\in B_1$ and the corresponding
$\mathbb{B}_z$, $\mathcal{U}_z$ and $K_z$ given by definition
\ref{def}. By Lemma \ref{KUAPL} $\mathcal{U}_z$ and $K_z$ enjoy the
 APP property. Thus we can apply Theorem 5.1 to these sets,  the domain $\Omega$
  equals $\mathbb{T}^2\times (0,T)$ and $A$ equals $(0,0,z)$.

We claim that each of the functions $U= (\rho,v,q) \in X$ with $U(x,t) \in K_z$ a.e
$(x,t)\in \Tp^2\times [0,T]$ from Theorem 5.1. yields a weak
solution $(\rho,v)$ to the IPM system with initial
values $\rho(x,0)=0$. Indeed from the fact that
$\mathcal{L}(\rho,v,q)=0$ it follows that $\textrm{div}(v)=0$ and
that $\textrm{Curl}(v + (0,\rho) )=0$. Thus there exists a
function $p$ such that $\nabla p=v+(0,\rho) $. The only delicate
issue corresponds to the equation of conservation of mass where we
see the nonlinearity. By definition we have that for $\psi \in
C_0^\infty [\Tp^2\times \R]$

\begin{equation*}\label{cons} 0=\int_{\Tp^2\times \R}( \partial_t \psi \rho+ \nabla \psi\cdot q).
\end{equation*}

Notice that $\varphi\in C_c^\infty [\Tp^2\times [0,T)]$  can be extended to $\tilde{\psi}\in C_0^\infty [\Tp^2\times \R]$.
Next we observe that  $q(x,t)=z$ for a.e. $(x,t)\in \Tp^2\times \R\setminus[0,T)$. Thus for a.e. $t\in\R\setminus[0,T)$ it holds that
\begin{equation}\label{constante}
\int_{\mathbb{T}^2} \nabla\tilde{\psi}(x,t)\cdot q(x,t) dx = \int_{\mathbb{T}^2} \nabla\tilde{\psi}(x,t)\cdot z dx =0.
\end{equation}
Now we plug $\tilde{\psi}$ in \eqref{cons}. Then \eqref{constante} together with the fact $\rho(x,t)=0$ for a.e. $(x,t)\in \Tp^2\times \R\setminus[0,T)$ yields
\[ 0=\int\limits_0^T\int\limits_{\Tp^2} (\partial_t \varphi \rho+
\nabla \varphi \cdot q)= \int\limits_0^T\int\limits_{\Tp^2}
(\partial_t \varphi \rho+ \nabla \varphi \cdot \rho  v),  \] where
in the last equality we have used that for $(x,t) \in\Tp^2 \times
[0,T],q=\rho v$. Hence  $(\rho,v)$ is a weak solution to the IPM
system according to \eqref{defweak}, \eqref{defweak2},
\eqref{defweak3}.
\end{proof}
\begin{remark}\label{irregularity}
The boundary values (or initial data) are attained in the weak sense. Namely we use
the fact (see \cite{DeLellisSzekelyhidi2})  that a solution to
$\mathcal{L}(U)=0$ can be redefined in a set of times of measure
zero in a way that the map
\[ t \to \int_{\mathbb{T}^2} v(x,t) \varphi \]
is continuous.  Since for negative $t$ $\int_{\mathbb{T}^2} v(x,t)
\varphi=0$ we obtain the desired
$$\lim_{t \to 0} \int_{\mathbb{T}^2} v(x,t) \varphi=0$$
and we can argue similarly for the density. From this point of
view it is not surprising that $\lim_{t \to 0} \rho(x,t)=0=
\lim_{t \to 0} v(x,t)$ but $\lim_{t \to 0} \rho(x,t)v(x,t)=z
\neq 0$ since they are weak limits and weak limits do not commute
with products. In fact from this we deduce that our solutions do
not attain the boundary values in the strong sense and hence by
the Frechet Kolmogorov theorem they are highly irregular. Another additional feature
of the weak solution is that $|\rho(x,t)|=1$ for a.e. $(x,t)\in \Tp^2\times (0,T)$
which prevents the pointwise convergence to the initial data.

\end{remark}

We conclude with a proposition pointing out to the fact that this
method does not seem to yield a solution compact in space and
time. For that we need some more standard terminology. We say that
a function $f:\mathbb{U}  \to \R$ is $\Lambda$ convex if for
$A,B,C \in \mathbb{U}$ such that $A-B \in \Lambda$ the function
$g(t):f(C+t(A-B))$ is convex. Let $K \subset \mathbb{U} $. Then
the  Lambda convex hull of $K$, $K^\lambda$

\[K^{\Lambda}=\{A \in \mathbb{U}: f(A) \le \sup f(C), \quad
C \in K, f \quad \Lambda {\textrm{ convex }}   \}.\]

In all the methods to solve differential inclusions  the canonical
choice for the set $\U$ is the interior of $K^\Lambda$. However in
our case the set $K$ can not be there.

\begin{Prop} \label{boundary}
The set $K \subset \partial {K^\Lambda}$ \end{Prop}
\begin{proof}

The proof relies on the fact  that the function $f(A)=z_2-\rho
v_2$ is $\Lambda$ convex. In fact it is easy to see that
$\det_{\Lambda}$ is $\Lambda$ linear, thus $-\det_{\Lambda}$ is
$\Lambda$ convex. Hence
\[ |v|^2-\det_{\Lambda}=-\rho v_2 \]
is $\Lambda$ convex and $f$ is $\Lambda$ convex. Moreover $f(A)=0$
for all $A \in K$ and hence $K^\Lambda \subset f^{-1}(-\infty,0]$
and  $K \subset \partial f^{-1}(-\infty,0] \cap K^\Lambda \subset
\partial K^\Lambda$.
\end{proof}
\begin{remark}
The above argument also shows that if $0 \in K^\Lambda$ (or any
other element in $K$) the elements of the splitting sequence in
the laminate \cite{Kirchheim03}  should belong to $f^{-1}(0)$.
\end{remark}

\subsection*{{\bf Acknowledgments}}

\smallskip

DC and FG were partially supported by {\sc MTM2008-03754} project of the MCINN (Spain) and StG-203138CDSIF grant of the ERC. DF was partially supported by MTM2008-02568 project of the MCINN (Spain). FG was partially supported by NSF-DMS grant 0901810. The authors would like to thank L. Sz\'ekelyhidi Jr for helpful conversations.

\bibliographystyle{acm}   

\def\cprime{$'$}


\end{document}